\documentclass[11pt]{amsart}
\usepackage[top=1.5in, bottom=1.5in, left=1in, right=1in]{geometry}   
\usepackage{amsmath, amsfonts, amssymb,amsthm}
\usepackage[pagebackref]{hyperref}

\usepackage{xcolor}
\usepackage{hyperref}
\hypersetup{
	colorlinks,
	linkcolor={red!50!black},
	citecolor={blue!62!black},
	urlcolor={blue!80!black}
}
\usepackage{enumitem}
\usepackage{cases}
\newtheorem*{corollary}{Corollary}
\theoremstyle{plain}
\newcommand{\thistheoremname}{}
\newtheorem*{genericthm*}{\thistheoremname}
\newenvironment{namedthm*}[1]{\renewcommand{\thistheoremname}{#1}
	\begin{genericthm*}}
	{\end{genericthm*}}
\theoremstyle{remark}

\DeclareMathOperator{\ord}{ord}
\title[Hyperbolicity of Fermat-type curves and thier complements]{Hyperbolicity of Fermat-type curves\\and their complements}
\subjclass[2010]{32H25, 32H30}
\keywords{Nevanlinna theory, complex hyperbolicity, Borel Theorem}

\author{Tuan Anh Nguyen}
\address{Department of Mathematics, University of Education, Hue University, 34 Le Loi St., Hue City, Vietnam}
\email{natuan@dhsphue.edu.vn}
\date{}
\begin{document}
    \begin{abstract}
        In this paper, by using the generalized Borel theorems in $\mathbb{CP}^2$, we show the hyperbolicity of Fermat type curves and their complement in $\mathbb{CP}^2$. This improves Noguchi-Shirosaki's and Demailly-El Goul's degree bounds.
    \end{abstract}
    \maketitle
%--------------------------------------
\section{Introduction}
    A complex manifold $X$ is said to be \textit{hyperbolic} (in the sense of \textit{Brody}) if it contains no entire curves, i.e., no image of a non-constant holomorphic map $f:\mathbb{C}\rightarrow X$. In 1970, Kobayashi conjectured that a generic hypersurface $D$ of degree $d$ large enough in the projective space $\mathbb{CP}^n$ is hyperbolic \cite{Kobayashi1970}.  According to Zaidenberg \cite{Zaidenberg1987}, the optimal bound for the degree is expected to be
    $d=2n-1$. In the so-called \textit{logarithmic} case, the complement of the generic hypersurface $D$ of degree $d\geq 2n+1$ in $\mathbb{CP}^n$ is also anticipated to be hyperbolic. The concept of hyperbolicity has attracted much attention since it is believed to intimately related to many areas of mathematics, especially in Diophantine approximation. For instance, the Lang-Vojta conjecture \cite{lang1991number,vojta1987diophantine} states that an algebraic variety $X$ over a number field $K$ contains only finitely many $K$-rational points if $X_\mathbb{C}$ is hyperbolic after some base change $K\hookrightarrow\mathbb{C}$. 

    In the recent years, significant progress has been made toward these conjectures. Notably, Kobayashi's conjecture was confirmed in cases where the degree of the hypersurface is very high compared to the dimension of the projective space. In the case of $\mathbb{CP}^3$, the first proof was given by McQuillan \cite{Mcquillan1999} with the degree $d\geq 36$. Later, the result was improved by Demailly-El Goul \cite{Demailly_Goul2000} and P\v{a}un \cite{mihaipaun2008}. In the case of $\mathbb{CP}^4$, the hyperbolicity was confirmed for generic hypersurfaces of degree $d\geq 593$ \cite{Rousseau2007,Diverio-Trapani2010}. In the case of arbitrary dimension $n$, the conjecture was verified with the sufficiently large degree
    $d(n)\gg 1$ \cite{Siu2004,berczi2024non,brotbek2017hyperbolicity}. 
    
    Regarding the logarithmic case, the hyperbolicity of the complement of the generic smooth curve of very high degree in $\mathbb{CP}^2$ was verified by Siu-Yeung \cite{siu1996hyperbolicity}. Later,  
    a lower degree bound $d\geq 14$ was given by Rousseau \cite{rousseau2009logarithmic}. In arbitrary dimension $n$, the conjecture was confirmed by Brotbek-Deng when the degree is sufficiently large \cite{brotbek2019kobayashi}.
    
    On the other hand, over the past five decades, constructing examples of hyperbolic hypersurfaces of low degree in $\mathbb{CP}^n$ has been considered a very challenging problem. The first example of hyperbolic surface in $\mathbb{CP}^3$ was given by Brody-Green \cite{Brody-Green1977}.
    In the logarithmic case, an example of a non-singular quintic curve in $\mathbb{CP}^2$ whose complement is hyperbolic was given by Zaidenberg \cite{zauidenberg1989stability}. Currently, there are two main methods to construct hyperbolic hypersurfaces in projective space. The first one is the deformation method introduced by Zaidenberg \cite{Zaidenberg1988}. Adapting this approach, Duval constructed a hyperbolic sextic surface in $\mathbb{CP}^3$ \cite{duval2003sextique}, which remains the best result in $\mathbb{CP}^3$ up to now. After that, Huynh followed up and obtained hyperbolic hypersurfaces of degree $d=2n$ in $\mathbb{CP}^n$ for $3<n\leq 6$ \cite{huynh2015examples}. For arbitrary dimension $n$, degree bounds $d\geq 16(n-1)^2$ and $d\geq 4(n-1)^2$ were given by Siu-Yeng \cite{siu_yeung1997} and by Shiffman-Zaidenberg \cite{shiffman_zaidenberg2002_pn} respectively; and the best result up to now is $d\geq\left[\frac{n+2}{2}\right]^2$ \cite{Huynh2016}.
    
    The second method is finding hyperbolic hypersurfaces among the class of perturbations of Fermat-type hypersurfaces. This method makes use of some variants of the generalized Borel theorem to study the degeneracy of holomorphic curves that either avoid or are contained in the Fermat-type hypersurfaces. Up to now, several variants have been established in \cite{siu_yeung1997,demailly1997connexions,huynh2024some}.
    In this paper, we aim to apply the generalized Borel theorems in \cite{huynh2024some} in the case of $\mathbb{CP}^2$ to improve the degree bound $d\geq 11$, previously obtained by Demailly-El Goul in \cite{demailly2012hyperbolic,demailly1997connexions}:
    \begin{namedthm*}
        {Theorem A}
        Consider the curve 
        \begin{equation*}
            \mathcal{C}=\{z_0^d+z_1^d+z_2^{d-2}(\varepsilon_0z_0^2+\varepsilon_1z_1^2+z_2^2)=0\}\subset\mathbb{CP}^2,
        \end{equation*}
        where $\varepsilon_0,\varepsilon_1$ satisfies:
        \renewcommand{\labelitemi}{-}
        \begin{itemize}
            \item $\varepsilon_0\lambda^2+\varepsilon_1\neq 0$ for all $\lambda$ satisfying $\lambda^d+1=0$; and
            \item $\varepsilon_0,\varepsilon_1$ are not solutions of $z(1-
            \frac{2}{d})\sqrt[d-2]{\frac{-2z}{d}}^2+1=0$.
        \end{itemize}
        Then, the complement $\mathbb{CP}^2\setminus\mathcal{C}$ is hyperbolic for all $d\geq 9$.
    \end{namedthm*}
    
    Regarding the Lang-Vojta conjecture, in the one-dimensional case, it reduces to the Mordell conjecture, which was proved in 1983 by Faltings \cite{faltings1983}. In higher dimensions, it remains an open problem. So far, only a few examples have been verified to satisfy the Lang-Vojta conjecture. The first such example was the family of hyperbolic hypersurfaces constructed by Shirosaki and Noguchi \cite{shirosaki1998some,noguchi2003arithmetic}. In particular, in \cite{shirosaki1998some}, Shirosaki proved that the family of Fermat-type curves 
    \begin{equation*}
        \mathcal{C}_{a,b}=\{(z_0:z_1:z_2)\in\mathbb{CP}^2:P(z_0,z_1)=aP(bz_1,z_2)\}\subset\mathbb{CP}^2
    \end{equation*}
    where $P(z_0,z_1)=z_0^d+z_1^d+z_0^ez_1^{d-e}$ is hyperbolic in case $a,b\neq 0,d>2e+8$ and $e>2$. Using this result, he then proved the hyperbolicity of hypersurfaces
    \begin{equation}
        \label{define hypersurface X_n}
        X_n=\{P_n(z_0,\dots,z_n)=0\}\subset\mathbb{CP}^n,
    \end{equation}
    where $P_n(z_0,\dots,z_n)$ are homogeneous polynomials defined inductively as: 
    \begin{align*}
        P_1(z_0,z_1)&=P(z_0,z_1),\\
        P_n(z_0,\dots,z_n)&=P_{n-1}(P(z_0,z_1),\dots,P(z_{n-1},z_n)).
    \end{align*}
    Following up, Noguchi proved the hyperbolicity of the curve $\mathcal{C}_{a,b}$ in case $b=0$ with the condition $e>3$ and, especially, verified that all hyperbolic hypersurfaces $X_n$ satisfy the Lang-Vojta conjecture \cite{noguchi2003arithmetic}. In our second result, by using the generalized Borel theorem in the compact case, we are able to relax the conditions of Shirosaki-Noguchi's theorem.
    \begin{namedthm*}{Theorem B}
        Let $a,b\in\mathbb{C}$ and $a\neq 0$. Then, the curve $\mathcal{C}_{a,b}\subset\mathbb{CP}^2$ defined by
        \begin{equation*}
            \mathcal{C}_{a,b}=\{(z_0:z_1:z_2)\in\mathbb{CP}^2\,|\,P(z_0,z_1)=aP(bz_1,z_2)\},
        \end{equation*}
        where $P(z_0,z_1)=z_0^d+z_1^d+z_0^ez_1^{d-e}$ is hyperbolic if
        \begin{enumerate}[label=(\roman*)]
            \item $b=0,d>e+3,e\geq 0$; or
            \item $b\neq 0,ab^d\not\in\{1,2\},d>2e+3,e\geq 0$; or
            \item $ab^d=2,d>2e+3,e>0$; or,
            \item $ab^d=1,d>2e+3,e\geq 0$ and $(d,e)=1$ if $e>0$.
        \end{enumerate}
    \end{namedthm*} 
    As a by-product, we improve the bound on the degree $d>12$ of the family of hyperbolic hypersurfaces satisfying the Lang-Vojta conjecture in \cite{noguchi2003arithmetic}:
    \begin{corollary}
        Under the same conditions as in \textbf{Theorem B}, the hyperbolic hypersurface $X_n$ defined as in \eqref{define hypersurface X_n} of degree $d^n$ with $d>5$ satisfies the Lang-Vojta's conjecture.
    \end{corollary}
%--------------------------------------
\section*{Acknowledgement}
    This research is funded by the University of Education, Hue University under grant number NCTBSV.T.25 TN - 101 - 02. The author would like to express his sincere gratitude to Dinh Tuan Huynh for his valuable guidance and support throughout the research. The author also appreciates the insightful comments and discussions from all members of Hue Geometry - Arithmetic Group, especially Dr. Nguyen Khanh Linh Tran and Dr. Ngoc Long Le. 
%--------------------------------------
\section{Generalized Borel theorems}
    In this section, we will present the generalized Borel theorems in both compact case and logarithmic case in \cite{huynh2024some} in case of $\mathbb{CP}^2$. Before entering the details, we briefly introduce Nevanlinna theory in the complex projective space. The notations and theorems are mainly referenced from \cite{huynh2024some} (see also \cite{noguchi2013nevanlinna} and \cite{ru2001nevanlinna} for further details).

    For any $r>0$, we denote $\Delta_r=\{z\in\mathbb{C}:|z|< r\}$ as the open disk of radius $r$ centered at the origin. Let $E=\sum_{i\in\mathbb{N}}\alpha_ia_i$ for $\alpha_i\geq 0,a_i\in\mathbb{C}$ be an effective divisor on $\mathbb{C}$. Then, for a truncation level $m\in\mathbb{N}\cup\{\infty\}$, we define the \textit{$m$-truncated degree} of $E$ on $\Delta_r$ as
    \begin{equation*}
        n^{[m]}(r,E)=\sum_{a_i\in\Delta_r}\min\{m,\alpha_i\}.
    \end{equation*}
    The \textit{truncated counting function at level $m$} of $E$ is then defined as
    \begin{equation*}
        N^{[m]}(r,E)=\int_1^r\frac{n^{[m]}(t,E)-n^{[m]}(0,E)}{t}\,dt\eqno{(r>1)}.
    \end{equation*}

    Let $f:\mathbb{C}\rightarrow\mathbb{CP}^n$ be a holomorphic map with a reduced representation $f=[f_0:\cdots:f_n]$ in the homogeneous coordinate $[z_0:\cdots:z_n]$ of $\mathbb{CP}^n$. Let $Q$ be a homogeneous polynomial of degree $d\geq 1$ and $D=\{Q=0\}$ be a hypersurface in $\mathbb{CP}^n$. If $f(\mathbb{C})\not\subset D$, then $f^*D=\sum_{a\in\mathbb{C}}(\ord_a Q\circ f)a$ is an effective divisor on $\mathbb{C}$. Then, we define \textit{the truncated counting function of $f$} with respect to $D$ is defined as
    \begin{equation*}
        N_f^{[m]}(r,D)=N^{[m]}(r,f^*D),
    \end{equation*}
    which measures how often $f(\Delta_r)$ intersects $D$. In case $m=\infty$, we denote $N_f(r,D)$ instead of $N_f^{[\infty]}(r,D)$. Next, the \textit{proximity function of $f$} associated with $D$ is defined as
    \begin{equation*}
        m_f(r,D)=\frac{1}{2\pi}\int_{0}^{2\pi}\log\frac{\left\|f(re^{i\theta})\right\|^d_{\max}\left\|Q\right\|_{\max}}{|(Q\circ f)(re^{i\theta})|}d\theta,
    \end{equation*}
    where 
    \begin{equation*}
        \left\|f(z)\right\|_{\max}=\max\{|f_0(z)|,\dots,|f_n(0)|\},
    \end{equation*}
    and $\left\|Q\right\|_{\max}$ is the maximum absolute value of the coefficients of $Q$. The proximity function measures how close $f(\Delta_r)$ is to $D$. We observe that $|Q\circ f|\leq\left(\substack{d+n\\ n}\right)\|Q\|_{\max}\|f\|^d_{\max}$, thus $m_f(r,D)\geq O(1)$. Finally, the \textit{the order function of $f$} is defined as
    \begin{equation*}
        T_f(r)=\frac{1}{2\pi}\int_{0}^{2\pi}\log\left\|f(re^{i\theta})\right\|_{\max}d\theta=\int_{1}^{r}\frac{dt}{t}\int_{|z|<t}f^*\omega_{FS}+O(1)\eqno{(r>1)},
    \end{equation*}
    measuring the growth of the area of $f(\Delta_r)$ with respect to the Fubini-Study metric $\omega_{FS}$. The Nevanlinna is established by comparing the three functions. It consists of two main theorems. The first one is a reformulation of the Poisson-Jensen formula.
    \begin{namedthm*}
        {First Main Theorem}
        Let $f:\mathbb{C}\rightarrow\mathbb{CP}^n$ be a holomorphic map and $D$ be a hypersurface of degree $d$ in $\mathbb{CP}^n$. If $f(\mathbb{C})\not\subset D$, then we have
        \begin{equation*}
            m_f(r,D)+N_f(r,D)=dT_f(r)+O(1)
        \end{equation*}
        for every $r>1$. Together with $m_f(r,D)\geq O(1)$, we have 
        \begin{equation*}
            N_f(r,D)\leq dT_f(r)+O(1).
        \end{equation*}
    \end{namedthm*}
    
    Hence, the First Main Theorem provides a lower bound of the order function in term of the counting function. The reverse dirction, which is called \textit{Second Main Theorem}, is usually more challenging and requires the sum of certain counting functions of many divisors. 

    A family $\{D_i\}_{0\leq i\leq q}$  of $q+1\geq n+1$ hypersurfaces in $\mathbb{CP}^n$ is said to be \textit{in general position} if $\cap_{i\in I}D_i=\varnothing$ for any subset $I\subset\{0,\dots,q\}$ of cardinality $n+1$. A holomorphic map $f$ is called \textit{linearly degenerate} if its image is contained in a hyperplane. For a linearly nondegenerate holomorphic curve $f:\mathbb{C}\rightarrow\mathbb{CP}^n$ and for a family of $q+1\geq n+2$ hyperplanes $\{H_i\}_{i\,=\,0,\dots,\, q}$ in general position, Cartan \cite{Cartan1933} established a Second Main Theorem
    \begin{equation*}
        (q-n)T_f(r)\leq\sum_{i=0}^qN^{[n]}_f(r,H_i)+S_f(r)\quad\|,
    \end{equation*}
    where the notation $\|$ means the inequality holds for every $r$ outside a subset of $\mathbb{R}^+$ of finite Lebesgue measure, and $S_f(r)$ is a real-value function satisfying
    \begin{equation*}
        \lim_{r\rightarrow\infty}\inf\frac{S_f(r)}{T_f(r)}=0.
    \end{equation*}
    In the case $n=1$, Cartan recovred the classical Second Main Theorem of Nevanlinna.

    By applying Cartan Second Main Theorem, one can obtain some generalizations of Borel theorem \cite{huynh2024some}. This is our main tools to study the degeneracy of curves. For completeness, we will present a detailed proof in the case of $\mathbb{CP}^2$.
    \begin{namedthm*}
        {Generalized Borel theorem in the logarithmic case}
        Let $d,\delta_0,\delta_1,\delta_2$ be non-negative integers such that $d>6+\sum_{i=0}^2\delta_i$. Let $Q_i$ $(i=0,1,2)$ be the homogeneous polynomial of degree $\delta_i$. Suppose that the family of curves $\{D_i\}_{i=0,1,2}$ where $D_i=\{z_i^{d-\delta_i}Q_i=0\}$ is in general position in $\mathbb{CP}^2$. Then, for the collection of $3$ holomorphic functions $f_0,f_1,f_2$ such that 
        \begin{equation*}
            f_0^{d-\delta_0}Q_0(f_0,f_1,f_2)+f_1^{d-\delta_1}Q_1(f_0,f_1,f_2)+f_2^{d-\delta_2}Q_2(f_0,f_1,f_2)
        \end{equation*}
        is nowhere vanishing, the indices set $\{0,1,2\}$ is partitioned into one of the following cases
        \begin{align*}
            \{0,1,2\}=I_0\cup I_1\text{ such that }|I_0|&=2,|I_1|=1;\text{ or }\\
            \{0,1,2\}=I_0\cup I_1\text{ such that }|I_0|&=1,|I_1|=2;\text{ or }\\
            \{0,1,2\}=I_0\cup I_1\text{ such that }|I_0|&=0,|I_1|=3;\text{ or }\\
            \{0,1,2\}=I_0\cup I_1\cup I_2\text{ such that }|I_0|&=0,|I_1|=1,|I_2|=2,
        \end{align*}  
        which satisfies the following conditions:
        \begin{enumerate}[label=(\roman*)]
            \item $f_i^{d-\delta_i}Q_i(f_0,f_1,f_2)\equiv 0$ if and only if $i\in I_0$;
            \item For any $\alpha>0$, $\frac{f_i^{d-\delta_i}Q_i(f_0,f_1,f_2)}{f_j^{d-\delta_j}Q_j(f_0,f_1,f_2)}$ is constant for all $i,j\in I_\alpha$;
            \item $\sum_{i\in I_\alpha}f_i^{d-\delta_i}Q_i(f_0,f_1,f_2)\equiv 0$ for any $\alpha\neq 1$.
        \end{enumerate}
    \end{namedthm*}
    \begin{proof}
        Firstly, we define function $f=[f_0:f_1:f_2]:\mathbb{C}\rightarrow\mathbb{CP}^2$ and set 
        \begin{align*}
            \pi\colon\mathbb{CP}^2&\rightarrow\mathbb{CP}^2,\qquad [z_0:z_1:z_2]\mapsto [z_0^{d-\delta_0}Q_0:z_1^{d-\delta_1}Q_1:z_2^{d-\delta_2}Q_2],\\
            g=[g_0:g_1:g_2]\colon \mathbb{C}&\rightarrow\mathbb{CP}^2,\qquad z\mapsto
            \pi\circ f(z),
        \end{align*}
        where $g_i:=(z_i^{d-\delta_i}Q_i)\circ f$. Let $\{H_i\}_{0\leq i\leq 3}$ be the family of $4$ projective lines in $\mathbb{CP}^{2}$ given by 
        \begin{align*}
            H_i&=\{z_i=0\}\qquad(0\leq i\leq 2),\\
            H_{3}&=\{\sum_{i=0}^{n}z_i=0\},
        \end{align*}
        which is in general position. Let $D$ be the curve defined by the equation
        \begin{equation*}
            z_0^{d-\delta_0}Q_0+z_1^{d-\delta_1}Q_1+z_2^{d-\delta_2}Q_2=0.
        \end{equation*}
        By our assumption, the image of $f$ avoids the curve $D$.
        Then, let $I_0=\{i\in\{0,1,2\}:g_i\equiv 0\}$. If $|I_0|=2$ then for $I_1=\{0,1,2\}\setminus I_0$, all conditions are satisfied. If $|I_0|=1$, say $I_0=\{0\}$, then the image of $g$ lies in the projective line $H_0\cong\mathbb{CP}^1$. Suppose that $g$ is non-constant, then applying the Second Main Theorem of Nevanlinna for $g$ and 3 distinct points $H_0\cap H_i$ ($i=1,2,3$), we have
        \begin{equation}
            \label{Cartan SMT for g}
            T_g(r)\leq\sum_{i=1}^3N^{[1]}_g(r,H_i)+S_g(r)\quad\|.
        \end{equation}

        Since $D_i$ ($i=0,1,2$) is general position, for any $w\in S_3=\{z\in\mathbb{C}^3:|z|=1\}$, there exists an index $i\in\{0,1,2\}$ such that $Q(f(w))\neq 0$. Since $\mathbb{S}_3$ is compact, there are $0<C_1\leq C_2$ such that
        \begin{equation*}
            C_1\leq \max_{i=0,1,2}\big|(z_i^{d-\delta_i}Q_i)(w)\big|^{\frac{1}{d}}\leq C_2\eqno(\forall\,w\in\,\mathbb{S}_{3}).
        \end{equation*}
        Since the polynomial $z_i^{d-\delta_i}Q_i$ is homogeneous, we have
        \begin{equation*}
            C_1\leq \max_{i=0,1,2}\frac{\big|(z_i^{d-\delta_i}Q_i)(w)\big|^{\frac{1}{d}}}{\|w\|}\leq C_2\eqno(\forall\,w\in\,\mathbb{C}^3).
        \end{equation*}
        It follows that 
        \begin{equation*}
            C_1\leq\max_{i=0,1,2}\frac{|g_i(z)|^\frac{1}{d}}{\|f(z)\|}\leq C_2\eqno(\forall z\in\mathbb{C}).
        \end{equation*}
        Thus, 
        \begin{equation*}
            T_g(r)=dT_f(r)+O(1).
        \end{equation*}
        Therefore, the estimate \eqref{Cartan SMT for g} becomes
        \begin{equation}
            \label{applying cartan to g and Hi, generalized borel}
            dT_f(r)\leq \sum_{i=1}^{2} N^{[1]}_f(r,D_i)+N^{[1]}_f(r,D)+S_f(r)\quad\|.
        \end{equation}
        For each $i\in\{0,1,2\}$, let $R_i$ be the curve defined by the polynomial $Q_i$. Then, we have
        \begin{equation*}
            N^{[1]}_f(r,D_i)=N^{[1]}_f(r,R_i)+\min\{1,d-
            \delta_i\}N^{[1]}_f(r,H_i).
        \end{equation*} 
        Hence, the estimate \eqref{applying cartan to g and Hi, generalized borel} becomes
        \begin{equation*}
            dT_f(r)\leq \sum_{i=1}^{2} [N^{[1]}_f(r,R_i)+N^{[1]}_f(r,H_i)]+N^{[1]}_f(r,D)+S_f(r)\quad\|.
        \end{equation*}
        Moreover, by the First Main Theorem, we have $N_f^{[1]}(r,R_i)\leq \delta_iT_f(r)+O(1)$ and $N_f^{[1]}(r,H_i)\leq T_f(r)+O(1)$.
        Therefore, we obtain
        \begin{equation*}
            dT_f(r)\leq \sum_{i=1}^{2} (\delta_i+1)\, T_f(r)+N^{[1]}_f(r,D)+S_f(r)\quad\|,
        \end{equation*}
        or equivalently
        \begin{equation}
            \label{dT_f(r) in terms of Nf}
            (d-(2+\delta_1+\delta_2))T_f(r)\leq N^{[1]}_f(r,D)+S_f(r)\quad\|.
        \end{equation}

        Since $f$ avoids $D$ and $d>2+\delta_1+\delta_2$, this estimate leads to a contradiction. Hence, $g$ is constant; so that $g_1/g_2$ is constant. Let $I_1=\{0,1,2\}\setminus I_0$, then all conditions are satisfied.
        
        In the last case, we consider $|I_0|=0$. Suppose that $g$ is linearly non-degenerate, then applying Cartan Second Main Theorem for $g$ and $4$ lines $H_i$ ($0\leq i\leq 3$), we have
        \begin{equation*}
            T_g(r)\leq\sum_{i=0}^3 N^{[2]}_g(r,H_i)+S_g(r)\quad\|.
        \end{equation*}
        By the same arguments as in the proof of the estimate \eqref{dT_f(r) in terms of Nf}, we obtain
        \begin{equation*}
            (d-(6+\sum_{i=0}^{2}\delta_i))T_f(r)\leq N^{[2]}_f(r,D) + S_f(r)\quad\|,
        \end{equation*}
        which leads to a contradiction since $f(\mathbb{C})$ avoids $D$. Hence, there exists a relation
        \begin{equation*}
            a_0g_0+a_1g_1+a_2g_2=0,
        \end{equation*}
        where $(a_0,a_1,a_2)\in\mathbb{C}^3\setminus\{(0,0,0)\}$. We observe that there is at most one coefficient $a_i=0$, otherwise $|I_0|>0$. If there exists one index $i$ such that $a_i=0$, say $i=0$, then we have
        \begin{equation*}
            a_1g_1+a_2g_2=0.
        \end{equation*} 
        Thus, $g_1/g_2$ is constant. Let $I_1=\{0\}$ and $I_2=\{0,1,2\}\setminus I_1$, we easily check that all conditions are satisfied. If $a_0,a_1,a_2\neq 0$, then we have $g:\mathbb{C}\mapsto H\cong\mathbb{CP}^1$ where $H=\{a_0z_0+a_1z_1+a_2z_2=0\}$. If $g$ is non-constant, then applying the Second Main Theorem for $g$ and 3 distinct points $H\cap H_i$ ($i=0,1,2$), we have
        \begin{equation*}
            T_g(r)\leq\sum_{i=0}^2N^{[1]}_g(r,H\cap H_i)+S_g(r)\quad\|.
        \end{equation*}
        By the same arguments as in the proof of estimate \eqref{dT_f(r) in terms of Nf}, we obtain that
        \begin{equation*}
            \big[d-(3+\sum_{i=0}^{2}\delta_i)\big]T_f(r)\leq S_f(r)\quad\|,
        \end{equation*}
        which contradicts the condition $d>6+\sum_{i=0}^{2}\delta_i$. Hence, $g$ is constant. Let $I_1=\{0,1,2\}$, then all conditions are satisfied.
    \end{proof}
    Similarly, we obtain an analog result in the compact case:
    \begin{namedthm*}
        {Generalized Borel theorem in the compact case} 
        Let $d,\delta_0,\delta_1,\delta_2$ be non-negative integers such that $d>3+\sum_{i=0}^2\delta_i$. Let $Q_i$ $(i=0,1,2)$ be the homogeneous polynomial of degree $\delta_i$. Suppose that the family of curves $\{D_i\}_{i=0,1,2}$ where $D_i=\{z_i^{d-\delta_i}Q_i=0\}$ is in general position in $\mathbb{CP}^2$. Then, for the collection of $3$ holomorphic functions $f_0,f_1,f_2$ such that
        \begin{equation*}
            f_0^{d-\delta_0}Q_0(f_0,f_1,f_2)+f_1^{d-\delta_1}Q_1(f_0,f_1,f_2)+f_2^{d-\delta_2}Q_2(f_0,f_1,f_2)\equiv 0,
        \end{equation*} 
        then the indices set $\{0,1,2\}$ is partitioned into one of the following cases
        \begin{align*}
            \{0,1,2\}=I_0\cup I_1\text{ such that }|I_0|&=1,|I_1|=2;\text{ or }\\
            \{0,1,2\}=I_0\cup I_1\text{ such that }|I_0|&=0,|I_1|=3,
        \end{align*}
        which satisfies the following conditions:
        \begin{enumerate}[label=(\roman*)]
            \item $f_i^{d-\delta_i}Q_i(f_0,f_1,f_2)\equiv 0$ if and only if $i\in I_0$;
            \item $\frac{f_i^{d-\delta_i}Q_i(f_0,f_1,f_2)}{f_j^{d-\delta_j}Q_j(f_0,f_1,f_2)}$ is constant for all $i,j\in I_1$;
            \item $\sum_{i\in I_\alpha}f_i^{d-\delta_i}Q_i(f_0,f_1,f_2)\equiv 0$ for all $\alpha$.
        \end{enumerate}
    \end{namedthm*}
    \begin{proof}
        We define the function $f=[f_0:f_1:f_2]:\mathbb{C}\rightarrow\mathbb{CP}^2$. Setting the function $g$ and the family of $4$ projective lines $\{H_i\}_{0\leq i\leq 3}$ as in the proof of the proof of the logarithmic case. Then, by the assumption, the image of $g$ lies in the projective line $H_3\cong\mathbb{CP}^1$. If $g$ is non-constant, then we apply the Second Main Theorem of Nevanlinna for $g$ and 3 distinct points $H_3\cap H_i$ ($i=0,1,2$), we have
        \begin{equation*}
            T_g(r)\leq\sum_{i=0}^2N^{[1]}_g(r,H_3\cap H_i)+S_g(r)\quad\|.
        \end{equation*}
        By the same argument as in the proof of the estimate \eqref{dT_f(r) in terms of Nf}, we obtain that
        \begin{equation*}
            \big[d-(3+\sum_{i=0}^{2}\delta_i)\big]T_f(r)\leq S_f(r)\quad\|,
        \end{equation*}
        which contradicts the condition $d>3+\sum_{i=0}^{2}\delta_i$. Hence, $g$ is constant. 

        Then, let $I_0=\{i\in\{0,1,2\}:g_i\equiv 0\}$ and $I_1=\{0,1,2\}\setminus I_0$. It is obvious that $|I_0|\leq 1$ (otherwise $g_i\equiv 0$ for all $i=0,1,2$), and $g_i/g_j$ is constant for all $i,j\in I_1$ (since $g$ is constant). Therefore, all conditions are satisfied. 
    \end{proof}
%--------------------------------------
\section{Hyperbolicity of complement of Noguchi-El Goul's curves in \texorpdfstring{$\mathbb{CP}^2$}{CP2}}
    In this section, we prove the \textbf{Theorem A}. 
    Firstly, let $f=[f_0:f_1:f_2]\colon\mathbb{C}\rightarrow\mathbb{CP}^2\setminus\mathcal{C}$ be a holomorphic map. Hence, 
        \begin{equation*}
            f_0^d+f_1^d+f_2^{d-2}(\varepsilon_0f_0^2+\varepsilon_1f_1^2+f_2^2)
        \end{equation*}
        is nowhere vanishing. We need to prove that $f$ is constant. 
        
        Let $Q_0=Q_1=1$ and $Q_2=\varepsilon_0z_0^2+\varepsilon_1z_1^2+z_2^2$. Thus, $\delta_0=\delta_1=0$ and $\delta_2=2$. It is clear that curves $D_i=\{z_i^{d-\delta_i}Q_i=0\}$ ($i\in\{0,1,2\}$) are in general position.
        Hence, with the condition
        \begin{equation*}
            d>6+2=8,
        \end{equation*}
        applying the generalized Borel theorem in the logarithmic case leads to four cases:
        \begin{enumerate}[label=(\roman*)]
            \item $|I_0|=2$ and $|I_1|=1$;
            \item $|I_0|=1$ and $|I_1|=2$;
            \item $|I_0|=0$ and $|I_1|=3$;
            \item $|I_0|=0, |I_1|=1$ and $|I_2|=2$.
        \end{enumerate}
        Let us consider the case (i). We have three subcases:
        \begin{enumerate}[label=(i.\arabic*)]
            \item $f_0^d=f_1^d\equiv 0$ and $f(\mathbb{C})$ avoids $z_2^{d-2}(\varepsilon_0z_0^2+\varepsilon_1z_1^2+z_2^2)=0$. Hence, $f_0=f_1\equiv 0$, which implies $f=[0:0:1]$ is constant.
            \item \label{subcase: El Goul 1.2}$f_0^d=f_2^{d-2}(\varepsilon_0f_0^2+\varepsilon_1f_1^2+f_2^2)\equiv 0$ and $f(\mathbb{C})$ avoids $z_1=0$. Thus, $f_0=0$ and $$f_2=0\text{ or }f_2=\lambda f_1,$$  
            where $\lambda^2+\varepsilon_1=0$. Hence, $f=[0:1:0]$ or $f=[0:1:\lambda]$, which means $f$ is constant.
            \item $f_1^d=f_2^{d-2}(\varepsilon_0f_0^2+\varepsilon_1f_1^2+f_2^2)\equiv 0$ and $f(\mathbb{C})$ avoids $z_0=0$. By symmetry, it can be treated similarly to the subcase \ref{subcase: El Goul 1.2}.
        \end{enumerate}
        In the case (ii), we have three subcases:
        \begin{enumerate}[label=(ii.\arabic*)]
            \item \label{subcase: El Goul 2.1} $f_0^d\equiv 0$ and $f_2^{d-2}(\varepsilon_0f_0^2+\varepsilon_1f_1^2+f_2^2)=\lambda f_1^d$ for some $\lambda\neq -1,0$. Thus, $f_0\equiv 0$ and $f_2^d+\varepsilon_1f_2^{d-2}f_1^2-\lambda f_1^d=0$,
            which implies $f_2=\mu f_1\text{ where }\mu^d+\varepsilon_1\mu^{d-2}-\lambda=0$.
            Hence, $f=[0:1:\mu]$ is constant.
            \item $f_1^d\equiv 0$ and $f_2^{d-2}(\varepsilon_0f_0^2+\varepsilon_1f_1^2+f_2^2)=\lambda f_0^d$ for some $\lambda\neq -1,0$. By symmetry, it can be treated similarly to the subcase \ref{subcase: El Goul 2.1}.
            \item $f_2^{d-2}(\varepsilon_0f_0^2+\varepsilon_1f_1^2+f_2^2)\equiv 0$ and $f_0^d=\lambda f_1^d$ for some $\lambda\neq -1,0$. The latter equality implies $f_0=\mu f_1$ where $\mu^d=\lambda$. Then, the former equation leads to $f_2\equiv 0\text{ or }f_2^2+(\varepsilon_0\mu^2+\varepsilon_1)f_1^2=0$.
            If $f_2\equiv 0$, then $f=[\mu:1:0]$ is constant. If $f_2^2+(\varepsilon_0\mu^2+\varepsilon_1)f_1^2=0$, then $f_2=\tau f_1$ where $\tau^2+(\varepsilon_0 \mu^2+\varepsilon_1)=0$. Hence, $f=[\mu:1:\tau]$ is constant.    
        \end{enumerate}
        As for (iii), we have
        \begin{equation*}
            \left\{
            \begin{array}{l}
                f_1^d=c_1f_0^d\\
                f_2^{d-2}(\varepsilon_0f_0^2+\varepsilon_1f_1^2+f_2^2)=c_2f_0^d,
            \end{array}
            \right.
        \end{equation*}
        where $c_1,c_2\neq 0$ and $1+c_1+c_2\neq 0$. Thus, $f_1=\lambda f_0$ where $\lambda^d=c_1$; and
        \begin{equation*}
            f_2^d+(\varepsilon_0+\varepsilon_1\lambda^2)f_2^{d-2}f_0^2-c_2f_0^d=0, 
        \end{equation*}
        which implies $f_2=\mu f_0$ where $\mu^d+(\varepsilon_0+\varepsilon_1\lambda^2)\mu^{d-2}-c_2=0$. Therefore, $f=[1:\lambda:\mu]$ is constant.\\
        In the last case (iv), we have three subcases:
        \begin{enumerate}[label=(iv.\arabic*)]
            \item $f_0^d+f_1^d=0$ and $f(\mathbb{C})$ avoids the curve $z_2^{d-2}(\varepsilon_0z_0^2+\varepsilon_1z_1^2+z_2^2)=0$. Hence, $f(\mathbb{C})$ lies in the projective space $z_0=\lambda z_1$ where $\lambda^d+1=0$. But this line intersects with the curve $z_2^{d-2}(\varepsilon_0z_0^2+\varepsilon_1z_1^2+z_2^2)=0$ at three distinct points $(\lambda:1:0),(\lambda:1:\mu),(\lambda:1:\tau)$ where $\mu,\tau$ are two distinct non-zero solutions of $(\varepsilon_0\lambda^2+\varepsilon_1)+z^2=0$ (since $\varepsilon_0\lambda^2+\varepsilon_1\neq 0$ by our assumptions). Therefore, $f$ must be constant by Little Picard theorem.
            \item \label{subcase: El Goul 4.2} The image of $f$ lies in the curve $z_0^d+z_2^{d-2}(\varepsilon_0z_0^2+\varepsilon_1z_1^2+z_2^2)=0$ and avoids the line $z_1=0$. The curve $z_0^d+z_2^{d-2}(\varepsilon_0z_0^2+\varepsilon_1z_1^2+z_2^2)=0$ in the inhomogeneous coordinates $(X,Y)$ is given by the equation 
            \begin{equation*}
                X^d+\varepsilon_0X^2+\varepsilon_1Y^2+1=0.
            \end{equation*}
            To find the singularities of this curve, we consider the system
            \begin{equation*}
            \left\{
            \begin{array}{l}
                X^d+\varepsilon_0X^2+\varepsilon_1Y^2+1=0\\
                dX^{d-1}+2\varepsilon_0X=0\\
                2\varepsilon_1Y=0.
            \end{array}
            \right.
            \end{equation*}
            Hence, we have
            \begin{equation*}
                \left\{
                    \begin{array}{l}
                        \varepsilon_0\left(1-\frac{2}{d}\right)X^2+1=0\\
                        X=0\text{ or } X=\sqrt[d-2]{\frac{-2 \varepsilon_0}{d}}\\
                        Y=0.
                    \end{array}
                    \right.
            \end{equation*}
            This system has no solutions under our assumptions. Hence, it is a smooth curve of the genus $g=\frac{(d-1)(d-2)}{2}\geq 28>2$, which implies $f$ is constant.
            \item The image of $f$ lies in the curve $z_1^d+z_2^{d-2}(\varepsilon_0z_0^2+\varepsilon_1z_1^2+z_2^2)=0$ and avoids the line $z_0=0$. By symmetry, it can be treated similarly to the subcase \ref{subcase: El Goul 4.2}.
        \end{enumerate} 
    \hfill$\qed$
    
    \section{Hyperbolicity of Noguchi-Shirosaki's Fermat-type curves}
    
    This section is devoted to prove the \textbf{Theorem B}.
    We start with the case $b=0$.
    Let $f=[f_0:f_1:f_2]\colon\mathbb{C}\rightarrow\mathcal{C}_{a,0}$ be a holomorphic map. Hence,
    \begin{equation*}
        P(f_0,f_1)=aP(0,f_2),
    \end{equation*}
    or equivalently
    \begin{equation*}
        f_0^d+f_1^{d-e}(f_0^e+f_1^e)-af_2^d=0.
    \end{equation*}
    We need to prove $f$ is constant. Let $Q_0=1,Q_1=z_0^e+z_1^e$ and $Q_2=-a$. Thus, $\delta_0=\delta_2=0$ and $\delta_1=e$. It is clear that curves $D_i=\{z_i^{d-\delta_i}Q_i=0\}$ ($i\in\{0,1,2\}$) are in general position.
    Then, with the condition
    \begin{equation*}
        d>3+e,
    \end{equation*}
    applying the generalized Borel theorem in the compact case, we have two cases:
    \begin{enumerate}[label=(\roman*)]
        \item $|I_0|=1$ and $|I_1|=2$;
        \item $|I_0|=0$ and $|I_1|=3$.
    \end{enumerate}
    In the case (i), we have three subcases:
    \begin{enumerate}[label=(i.\arabic*)]
        \item $f_0^d\equiv 
        0$ and $f_1^{d-e}(f_0^e+f_1^e)-af_2^d=0$. Hence, $f_0=0$ and $f_1^d-af_2^d=0$, which implies that $f_1=\lambda f_2\text{ where } \lambda^d-a=0$.
        Thus, $f=[0:\lambda:1]$ is constant.
        \item $f_1^{d-e}(f_0^e+f_1^e)=0$ and $f_0^d-af_2^d=0$. The former equality implies that $f_1\equiv 0$ or $f_1=\lambda f_0$ where $\lambda^e+1=0$; and the latter equation leads to $f_2=\mu f_0$ where $1-a\mu^d=0$. Thus, $f=[1:0:\mu]$ or $f=[1:\lambda:\mu]$, which means $f$ is constant.  
        \item $-af_2^d\equiv 0$ and $f_0^d+f_1^{d-e}(f_0^e+f_1^e)=0$. Thus, $f_2\equiv 0$ and $f_0=\lambda f_1$ where $\lambda^d+\lambda^e+1=0$. Hence, $f=[\lambda:1:0]$ is constant.
    \end{enumerate}
    As for (ii), we have
    \begin{equation*}
        \left\{
        \begin{array}{l}
            f_1^{d-e}(f_0^e+f_1^e)=c_1f_0^d\\
            -af_2^d=c_2f_0^d,
        \end{array}
        \right.
    \end{equation*}
    where $c_1,c_2\neq 0$ and $1+c_1+c_2=0$. Hence, we have
    \begin{equation*}
        \left\{
        \begin{array}{l}
            f_1=\lambda f_0\\
            f_2=\mu f_0,
        \end{array}
        \right.
    \end{equation*}
    where $\lambda^d+\lambda^{d-e}-c_1=0$ and $a\mu^d+c_2=0$. Therefore, $f=[1:\lambda:\mu]$ is constant.
    \\[1ex]

    Now, we turn to the case $b\neq0$.
    It is sufficient to consider $b=1$. Indeed, for any $b\neq 0$, we have
    \begin{align*}
        aP(bz_1,z_2)&=a(b^dz_1^d+z_2^d+b^ez_1^ez_2^{d-e})\\
        &=ab^d(z_1^d+(z_2/b)^d+z_1^e(z_2/b)^d)\\
        &=ab^dP(z_1,z_2/b).
    \end{align*} 
    Let $f=[f_0:f_1:f_2]\colon\mathbb{C}\rightarrow\mathcal{C}_{a,1}$ be a holomorphic map. Hence, 
    \begin{align*}
        P(f_0,f_1)=aP(f_1,f_2),
    \end{align*}
    or equivalently
    \begin{equation}
        \label{Noguchi-Shirosaki's curves when a=1,e>0}
        f_0^d+f_1^{d-e}(f_0^e+(1-a)f_1^e)-af_2^{d-e}(f_1^e+f_2^e)=0.
    \end{equation}
    
    We need to prove $f$ is constant. Let $Q_0=1,Q_1=z_0^e+(1-a)z_1^e$ and $Q_2=-a(z_1^e+z_2^e)$. Thus, $\delta_0=1$ and $\delta_1=\delta_2=e$. 
    Under one of the following conditions:
    \begin{itemize}
        \item $a\not\in\{1,2\},e\geq 0$; or
        \item $a=2,e>0$; or
        \item $a=1,e=0$,
    \end{itemize}
    the system
    \begin{equation*}
        z_0^d=z_1^{d-e}(z_0^e+(1-a)z_1^e)=-az_2^{d-e}(z_1^e+z_2^e)=0.
    \end{equation*}
    has only trivial solution. Thus,  $D_i=\{z_i^{d-\delta_i}Q_i=0\}$ ($i\in\{0,1,2\}$) are in general position. Then, with the condition
    \begin{equation*}
        d>3+2e,
    \end{equation*}
    applying the generalized Borel theorem in the compact case, we have two cases:
    \begin{enumerate}[label=(\roman*)]
        \item $|I_0|=1$ and $|I_1|=2$;
        \item $|I_0|=0$ and $|I_1|=3$.
    \end{enumerate}
    Firstly, we consider (i). We have three subcases:
    \begin{enumerate}[label=(i.\arabic*)]
        \item $f_0^d\equiv 0$ and $f_1^{d-e}(f_0^e+(1-a)f_1^e)-af_2^{d-e}(f_1^e+f_2^e)=0$. Hence, $f_0\equiv 0$ and $(1-a)f_1^d-af_1^ef_2^{d-e}-af_2^d=0$, which implies $f_2=\lambda f_1\text{ where } (1-a)-a\lambda^{d-e}-a\lambda^d=0$. Thus, $f=[0:1:\lambda]$ is constant.    
        \item $f_1^{d-e}(f_0^e+(1-a)f_1^e)=0$ and $f_0^d-af_2^{d-e}(f_1^e+f_2^e)=0$. The former equation implies $f_1\equiv 0$ or $f_0=\lambda f_1$ where $\lambda^e+(1-a)=0$. If $f_1=0$ then we have $f_0^d-af_2^d=0$, which implies $f_0=\mu f_2 \text{ where } \mu^d-a=0$. Thus, $f=[\mu:0:1]$ is constant. If $f_0=\lambda f_1$, then $\lambda^df_1^d-af_1^ef_2^{d-e}-af_2^d=0$, which implies $f_2=\mu f_1\text{ where }\lambda^d-a\mu^{d-e}-a\mu^d=0$. Thus, $f=[\lambda:1:\mu]$ is constant.
        \item $-af_2^{d-2}(f_1^e+f_2^e)=0$ and $f_0^d+f_1^{d-e}(f_0^e+(1-a)f_1^e)=0$. The latter equality implies $f_0=\lambda f_1$ where $\lambda^d+\lambda^e+(1-a)=0$; and the former equation implies $f_2=0$ or $f_2=\mu f_1$ where $\mu^e+1=0$. Hence, $f=[\lambda:1:\mu]$ or $f=[\lambda:1:0]$, which means $f$ is constant. 
    \end{enumerate}
    In the case (ii), we have
    \begin{equation*}
        \left\{
        \begin{array}{l}
            f_1^{d-e}(f_0^e+(1-a)f_1^e)=c_1f_0^d\\
            -af_2^{d-e}(f_1^e+f_2^e)=c_2f_0^d,
        \end{array}
        \right.
    \end{equation*}
    where $c_1,c_2\neq 0$ and $1+c_1+c_2=0$. Thus, $f_1=\lambda f_0$ where $(1-a)\lambda^d+\lambda^{d-e}-c_1=0$; and $af_2^d+a\lambda^ef_2^{d-e}f_0^e+c_2f_0^d=0$, which implies $f_2=\mu f_0$ where $a\mu^d+a\lambda^e\mu^{d-e}+c_2=0$. Therefore, $f=[1:\lambda:\mu]$ is constant. 

    Finally, we consider the case $1-a=0$ and $e>0$. Then, the equation \eqref{Noguchi-Shirosaki's curves when a=1,e>0} becomes
    \begin{equation*}
        f_0^d+f_0^ef_1^{d-e}-f_1^ef_2^{d-e}-f_2^d=0.
    \end{equation*} 
    Hence, the image of $f$ lies in the curve $z_0^d+z_0^ez_1^{d-e}-z_1^ez_2^{d-e}-z_2^d=0$. This curve in the inhomogeneous coordinates $(X,Y)$ is given by the equation
    \begin{equation*}
        X^d+X^e-Y^{d-e}-Y^d=0.
    \end{equation*}
    To find singular points of this curve, we consider the system
    \begin{equation*}
        \left\{
        \begin{array}{lr}
            X^d+X^e-Y^{d-e}-Y^d&=0\\
            dX^{d-1}+eX^{e-1}&=0\\
            (d-e)Y^{d-e-1}+dY^{d-1}&=0.
        \end{array}
        \right.
    \end{equation*}
    
    It implies that
    \begin{equation}
        \label{sys: equiv system}
        \left\{
            \begin{array}{l}
                X=0\text{ or } X^{d-e}=\frac{-e}{d}\\
                Y=0\text{ or } Y^e=-\frac{d-e}{d},
            \end{array}
        \right.
    \end{equation}
    and
    \begin{equation*}
        d(X^d+X^e-Y^{d-e}-Y^d)-X(dX^{d-1}+eX^{e-1})+Y((d-e)Y^{d-e-1}+dY^{d-1})=0,
    \end{equation*}
    or equivalently
    \begin{equation}
        \label{eqn: equiv curve}
        (d-e)X^e-eY^{d-e}=0.
    \end{equation}
    Thus, if $X=0$ then \eqref{eqn: equiv curve} implies $Y=0$ (and vice versa). We consider the case $XY\neq 0$. Then, the equation \eqref{eqn: equiv curve} becomes
    \begin{equation*}
        (d-e)X^{e-d}X^d = eY^{-e}Y^d.
    \end{equation*}
    Together with the system \eqref{sys: equiv system} and $XY\neq 0$, we obtain
    \begin{equation*}
        (d-e)\frac{-d}{e}X^d=e\frac{-d}{d-e}Y^d,
    \end{equation*}
    or equivalently
    \begin{equation*}
        \left(\frac{d-e}{e}\right)^2X^d=Y^d.
    \end{equation*}
    Taking the modulus both sides, we have
    \begin{equation*}
        \left(\frac{d}{e}-1\right)^2|X|^d=|Y|^d.
    \end{equation*}
    Moreover, the system \eqref{sys: equiv system} implies that $|X|=\left(\frac{e}{d}\right)^{\frac{1}{d-e}}$ and $|Y|=\left(\frac{d-e}{d}\right)^{\frac{1}{e}}$ (since $XY\neq 0$). Hence,
    \begin{equation*}
        \left(\frac{d}{e}-1\right)^2\left(\frac{e}{d}\right)^{\frac{d}{d-e}}=\left(\frac{d-e}{d}\right)^{\frac{d}{e}},
    \end{equation*}
    or equivalently
    \begin{equation}
        \label{equivalent equation of non-zero singular point}
        \left(\frac{d}{e}-1\right)^2\left(\frac{d}{e}\right)^{-1-\frac{e}{d-e}}=\left(1+\frac{e}{d-e}\right)^{-\frac{d}{e}}.
    \end{equation}
    Since $d>2e+3$, we have $\frac{d}{e}>\frac{2e+3}{3}>2$ and $-1-\frac{e}{d-e}>-1-\frac{e}{e+3}>-2$. Thus,
    \begin{equation*}
        \left(\frac{d}{e}-1\right)^2\left(\frac{d}{e}\right)^{-1-\frac{e}{d-e}}>\left(\frac{e+3}{e}\right)^2\left(\frac{2e+3}{e}\right)^{-2}=\left(\frac{e+3}{2e+3}\right)^2.
    \end{equation*}
    Other hand, $1<1+\frac{e}{d-e}<1+\frac{e}{e+3}=\frac{2e+3}{e+3}$ and $-\frac{d}{e}<-\frac{2e+3}{e}<-2$. Hence
    \begin{equation*}
        \left(1+\frac{e}{d-e}\right)^{-\frac{d}{e}}<\left(\frac{2e+3}{e+3}\right)^{-2}.
    \end{equation*}     
    It implies that 
    \begin{equation*}
        \left(1+\frac{e}{d-e}\right)^{-\frac{d}{e}}<\left(\frac{e+3}{2e+3}\right)^2<\left(\frac{d}{e}-1\right)^2\left(\frac{d}{e}\right)^{-1-\frac{e}{d-e}},
    \end{equation*}
    which contradicts \eqref{equivalent equation of non-zero singular point}.
    Therefore, $(0,0)$ is the unique singular point of the curve $X^d+X^e-Y^{d-e}-Y^d=0$. By a Taylor expansion at $(0,0)$, the curve around $(0,0)$ is given locally by $X^e-Y^{d-e}=0$. Toghether with the condition $(d,e)=1$, this curve is irreducible. It follows that 
    (cf. \cite{goul1996algebraic}) 
    the genus of this curve is
    \begin{equation*}
        g=\frac{(d-1)(d-2)}{2}-\delta_0,
    \end{equation*}
    where 
    \begin{equation*}
        \delta_0=\frac{(e-1)(d-e-1)}{2}+\frac{\gcd(e,d-e)-1}{2}.
    \end{equation*}
    Hence
    \begin{align*}
        g&=\frac{d^2-(e+2)d+e^2+1}{2}.
    \end{align*}
    Since $d>2e+3>\frac{e+2}{2}$, we obtain
    \begin{align*}
        g&>\frac{(2e+3)^2-(e+2)(2e+3)+e^2+1}{2}\\
        &=\frac{3e^2+5e+4}{2}\\
        &>2.
    \end{align*}
    Therefore, $f$ is constant.

    \hfill$\qed$

%--------------------------------------
    \bibliographystyle{plain}
    \bibliography{references}

@article{noguchi2003arithmetic,
  title={An arithmetic property of Shirosakis hyperbolic projective hypersurface},
  author={Noguchi, Junjiro},
  journal={de Gruyter},
  year={2003},
  publisher={Walter de Gruyter GmbH \& Co. KG Berlin, Germany}
}

@article{demailly2012hyperbolic,
  title={Hyperbolic algebraic varieties and holomorphic differential equations},
  author={Demailly, Jean-Pierre},
  journal={Acta Math. Vietnam},
  volume={37},
  number={4},
  pages={441--512},
  year={2012}
}

@article{huynh2024some,
  title={Some variants of the generalized Borel Theorem and applications},
  author={Huynh, Dinh Tuan},
  journal={arXiv preprint arXiv:2407.16163},
  year={2024}
}

@book {Kobayashi1970,
    AUTHOR = {Kobayashi, Shoshichi},
     TITLE = {Hyperbolic manifolds and holomorphic mappings},
    SERIES = {Pure and Applied Mathematics},
    VOLUME = {2},
 PUBLISHER = {Marcel Dekker, Inc., New York},
      YEAR = {1970},
     PAGES = {ix+148},
   MRCLASS = {32.60},
  MRNUMBER = {0277770 (43 \#3503)},
MRREVIEWER = {W. Kaup},
}

@article {Zaidenberg1987,
    AUTHOR = {Zaidenberg, Mikhail},
     TITLE = {The complement to a general hypersurface of degree {$2n$} in
              {${\bf CP}^n$} is not hyperbolic},
   JOURNAL = {Sibirsk. Mat. Zh.},
  FJOURNAL = {Akademiya Nauk SSSR. Sibirskoe Otdelenie. Sibirski\u\i\
              Matematicheski\u\i\ Zhurnal},
    VOLUME = {28},
      YEAR = {1987},
    NUMBER = {3},
     PAGES = {91--100, 222},
      ISSN = {0037-4474},
   MRCLASS = {32H15 (32H20)},
  MRNUMBER = {904640 (88k:32063)},
MRREVIEWER = {Valentin Zdravkov Hristov},
}

@book{lang1991number,
  author    = {Serge Lang},
  title     = {Number Theory III: Diophantine Geometry},
  series    = {Encyclopaedia of Mathematical Sciences},
  volume    = {60},
  publisher = {Springer-Verlag},
  year      = {1991}
}

@book{vojta1987diophantine,
  author    = {Paul Vojta},
  title     = {Diophantine Approximations and Value Distribution Theory},
  series    = {Lecture Notes in Mathematics},
  volume    = {1239},
  publisher = {Springer-Verlag},
  year      = {1987}
}

@article{duval2003sextique,
  title={Une sextique hyperbolique dans P\^{} 3 (C)},
  author={Duval, Julien},
  journal={arXiv preprint math/0312304},
  year={2003}
}

@article{huynh2015examples,
  title={Examples of hyperbolic hypersurfaces of low degree in projective spaces},
  author={Huynh, Dinh Tuan},
  journal={arXiv preprint arXiv:1507.03542},
  year={2015}
}

@article {Huynh2016,
    AUTHOR = {Huynh, Dinh Tuan},
     TITLE = {Construction of hyperbolic hypersurfaces of low degree in
              {$\Bbb{P}^n(\Bbb{C})$}},
   JOURNAL = {Internat. J. Math.},
  FJOURNAL = {International Journal of Mathematics},
    VOLUME = {27},
      YEAR = {2016},
    NUMBER = {8},
     PAGES = {1650059},
      ISSN = {0129-167X},
   MRCLASS = {32Q45 (14J70 14N25 32J25)},
  MRNUMBER = {3530280},
MRREVIEWER = {Simone Diverio},
       URL = {https://doi.org/10.1142/S0129167X16500592},
}

@article {Zaidenberg1988,
    AUTHOR = {Zaidenberg, Mikhail},
     TITLE = {Stability of hyperbolic embeddedness and the construction of
              examples},
   JOURNAL = {Mat. Sb. (N.S.)},
  FJOURNAL = {Matematicheski\u\i\ Sbornik. Novaya Seriya},
    VOLUME = {135(177)},
      YEAR = {1988},
    NUMBER = {3},
     PAGES = {361--372, 415},
      ISSN = {0368-8666},
   MRCLASS = {32H20 (14H99 14J99 32H15)},
  MRNUMBER = {937646},
MRREVIEWER = {J. T. Davidov},
}

@article{faltings1983,
  author    = {Gerd Faltings},
  title     = {Endlichkeitssätze für abelsche Varietäten über Zahlkörpern},
  journal   = {Inventiones mathematicae},
  volume    = {73},
  number    = {3},
  pages     = {349--366},
  year      = {1983},
  publisher = {Springer}
}

@article {Mcquillan1999,
    AUTHOR = {McQuillan, M.},
     TITLE = {Holomorphic curves on hyperplane sections of {$3$}-folds},
   JOURNAL = {Geom. Funct. Anal.},
  FJOURNAL = {Geometric and Functional Analysis},
    VOLUME = {9},
      YEAR = {1999},
    NUMBER = {2},
     PAGES = {370--392},
      ISSN = {1016-443X},
   MRCLASS = {32Q45 (14J30 14J70 32H30)},
  MRNUMBER = {1692470},
MRREVIEWER = {Atanas Iliev},
       DOI = {10.1007/s000390050091},
       URL = {https://doi.org/10.1007/s000390050091},
}

@article {Demailly_Goul2000,
    AUTHOR = {Demailly, Jean-Pierre and El Goul, Jawher},
     TITLE = {Hyperbolicity of generic surfaces of high degree in projective
              3-space},
   JOURNAL = {Amer. J. Math.},
  FJOURNAL = {American Journal of Mathematics},
    VOLUME = {122},
      YEAR = {2000},
    NUMBER = {3},
     PAGES = {515--546},
      ISSN = {0002-9327},
     CODEN = {AJMAAN},
   MRCLASS = {32Q45 (14J29 14J70)},
  MRNUMBER = {1759887 (2001f:32045)},
MRREVIEWER = {Tomasz Szemberg},
       URL = {http://muse.jhu.edu/journals/american_journal_of_mathematics/v122/122.3demailly.pdf},
}

@article {mihaipaun2008,
    AUTHOR = {P{\u{a}}un, Mihai},
     TITLE = {Vector fields on the total space of hypersurfaces in the
              projective space and hyperbolicity},
   JOURNAL = {Math. Ann.},
  FJOURNAL = {Mathematische Annalen},
    VOLUME = {340},
      YEAR = {2008},
    NUMBER = {4},
     PAGES = {875--892},
      ISSN = {0025-5831},
     CODEN = {MAANA},
   MRCLASS = {14J70 (32Q45)},
  MRNUMBER = {2372741 (2009b:14085)},
MRREVIEWER = {Erwan Rousseau},
        URL = {http://dx.doi.org/10.1007/s00208-007-0172-5},
}

@article {Rousseau2007,
    AUTHOR = {Rousseau, Erwan},
     TITLE = {Weak analytic hyperbolicity of generic hypersurfaces of high
              degree in {$\Bbb P^4$}},
   JOURNAL = {Ann. Fac. Sci. Toulouse Math. (6)},
  FJOURNAL = {Annales de la Facult\'e des Sciences de Toulouse.
              Math\'ematiques. S\'erie 6},
    VOLUME = {16},
      YEAR = {2007},
    NUMBER = {2},
     PAGES = {369--383},
      ISSN = {0240-2963},
   MRCLASS = {32Q45 (14J70 32H30)},
  MRNUMBER = {2331545 (2008e:32035)},
MRREVIEWER = {Yoshihiro Aihara},
       URL = {http://afst.cedram.org.sci-hub.org/item?id=AFST_2007_6_16_2_369_0},
}

@article {Diverio-Trapani2010,
    AUTHOR = {Diverio, Simone and Trapani, Stefano},
     TITLE = {A remark on the codimension of the {G}reen-{G}riffiths locus
              of generic projective hypersurfaces of high degree},
   JOURNAL = {J. Reine Angew. Math.},
  FJOURNAL = {Journal f\"ur die Reine und Angewandte Mathematik. [Crelle's
              Journal]},
    VOLUME = {649},
      YEAR = {2010},
     PAGES = {55--61},
      ISSN = {0075-4102},
     CODEN = {JRMAA8},
   MRCLASS = {14J70 (14C20 32Q45)},
  MRNUMBER = {2746466 (2012b:14086)},
MRREVIEWER = {Dmitry Kerner},
       URL = {http://dx.doi.org.sci-hub.org/10.1515/CRELLE.2010.088},
}

@incollection {Siu2004,
    AUTHOR = {Siu, Yum-Tong},
     TITLE = {Hyperbolicity in complex geometry},
 BOOKTITLE = {The legacy of {N}iels {H}enrik {A}bel},
     PAGES = {543--566},
 PUBLISHER = {Springer, Berlin},
      YEAR = {2004},
   MRCLASS = {32Q45 (14J29 14K05 32H30)},
  MRNUMBER = {2077584},
MRREVIEWER = {William A. Cherry},
}

@article{Brody-Green1977,
 author               = {Brody, Robert and Green, Mark},
 fjournal             = {Duke Mathematical Journal},
 issn                 = {0012-7094},
 journal              = {Duke Math. J.},
 mrclass              = {32H20},
 mrnumber             = {0454080 (56 \#12331)},
 mrreviewer           = {Hirotaka Fujimoto},
 number               = {4},
 pages                = {873--874},
 title                = {A family of smooth hyperbolic hypersurfaces in {$P_{3}$}},
 volume               = {44},
 year                 = {1977},
 }

@article{shirosaki1998some,
  title={On some hypersurfaces and holomorphic mappings},
  author={Shirosaki, Manabu},
  journal={Kodai Mathematical Journal},
  volume={21},
  number={1},
  pages={29--34},
  year={1998},
  publisher={Department of Mathematics, Tokyo Institute of Technology}
}

@article{brotbek2017hyperbolicity,
  title={On the hyperbolicity of general hypersurfaces},
  author={Brotbek, Damian},
  journal={Publications math{\'e}matiques de l'IH{\'E}S},
  volume={126},
  pages={1--34},
  year={2017}
}

@article{berczi2024non,
  title={Non-reductive geometric invariant theory and hyperbolicity},
  author={B{\'e}rczi, Gergely and Kirwan, Frances},
  journal={Inventiones mathematicae},
  volume={235},
  number={1},
  pages={81--127},
  year={2024},
  publisher={Springer}
}

@article{rousseau2009logarithmic,
  title={Logarithmic vector fields and hyperbolicity},
  author={Rousseau, Erwan},
  journal={Nagoya Mathematical Journal},
  volume={195},
  pages={21--40},
  year={2009},
  publisher={Cambridge University Press}
}

@article{siu1996hyperbolicity,
  title={Hyperbolicity of the complement of a generic smooth curve of high degree in the complex projective plane},
  author={Siu, Yum-Tong and Yeung, Sai-kee},
  journal={Inventiones mathematicae},
  volume={124},
  number={1},
  pages={573--618},
  year={1996},
  publisher={Springer}
}

@article{brotbek2019kobayashi,
  title={Kobayashi hyperbolicity of the complements of general hypersurfaces of high degree},
  author={Brotbek, Damian and Deng, Ya},
  journal={Geometric and Functional Analysis},
  volume={29},
  pages={690--750},
  year={2019},
  publisher={Springer}
}

@article{zauidenberg1989stability,
  title={Stability of hyperbolic imbeddedness and construction of examples},
  author={Za{\u\i}denberg, MG},
  journal={Mathematics of the USSR-Sbornik},
  volume={63},
  number={2},
  pages={351},
  year={1989},
  publisher={IOP Publishing}
}

@book{noguchi2013nevanlinna,
  title={Nevanlinna theory in several complex variables and Diophantine approximation},
  author={Noguchi, Junjiro and Winkelmann, J{\"o}rg},
  volume={350},
  year={2013},
  publisher={Springer Science \& Business Media}
}

@article {siu_yeung1997,
    AUTHOR = {Siu, Yum-Tong and Yeung, Sai-Kee},
     TITLE = {Defects for ample divisors of abelian varieties, {S}chwarz
              lemma, and hyperbolic hypersurfaces of low degrees},
   JOURNAL = {Amer. J. Math.},
  FJOURNAL = {American Journal of Mathematics},
    VOLUME = {119},
      YEAR = {1997},
    NUMBER = {5},
     PAGES = {1139--1172},
      ISSN = {0002-9327},
     CODEN = {AJMAAN},
   MRCLASS = {32H20 (32H30)},
  MRNUMBER = {1473072 (98h:32044)},
MRREVIEWER = {Min Ru},
       URL = {http://muse.jhu.edu/journals/american_journal_of_mathematics/v119/119.5siu.pdf},
}

@article {shiffman_zaidenberg2002_pn,
    AUTHOR = {Shiffman, Bernard and Zaidenberg, Mikhail},
     TITLE = {Hyperbolic hypersurfaces in {$\Bbb P\sp n$} of
              {F}ermat-{W}aring type},
   JOURNAL = {Proc. Amer. Math. Soc.},
  FJOURNAL = {Proceedings of the American Mathematical Society},
    VOLUME = {130},
      YEAR = {2002},
    NUMBER = {7},
     PAGES = {2031--2035},
      ISSN = {0002-9939},
     CODEN = {PAMYAR},
   MRCLASS = {32Q45 (14J70 32H25)},
  MRNUMBER = {1896038 (2003e:32044)},
MRREVIEWER = {Min Ru},
       URL = {http://dx.doi.org/10.1090/S0002-9939-01-06417-6},
}

@article{demailly1997connexions,
  title={Connexions m{\'e}romorphes projectives partielles et vari{\'e}t{\'e}s alg{\'e}briques hyperboliques},
  author={Demailly, Jean-Pierre and El Goul, Jawher},
  journal={Comptes Rendus de l'Acad{\'e}mie des Sciences-Series I-Mathematics},
  volume={324},
  number={12},
  pages={1385--1390},
  year={1997},
  publisher={Elsevier}
}

@article{goul1996algebraic,
  title={Algebraic families of smooth hyperbolic surfaces of low degree in PC3},
  author={Goul, Jawher El},
  journal={manuscripta mathematica},
  volume={90},
  pages={521--532},
  year={1996},
  publisher={Springer}
}

@book{ru2001nevanlinna,
  title={Nevanlinna theory and its relation to Diophantine approximation},
  author={Ru, Min},
  year={2001},
  publisher={World Scientific}
}

@article{Cartan1933,  
    title = {Sur les z\'{e}ros des combinaisons lin\'{e}aires de $p$ fonctions holomorphes donn\'{e}es},
    author = { Cartan, Henri},
    journal = {Mathematica},
    volume = {7},
    pages = {80--103},
    year = {1933},

}
\end{document}